\documentclass[11pt]{amsart}
 \usepackage{amsmath,amssymb}
 \usepackage{xcolor}
 \usepackage{hyperref}
\usepackage[scr=boondoxo,scrscaled=1.02]{mathalfa}
\newcommand{\cl}[1]{\mathscr{#1}}

\newtheorem{theorem}{Theorem}[section]

\newtheorem{lemma}[theorem]{Lemma}

\newtheorem{prop}[theorem]{Proposition}

\theoremstyle{definition}

\newtheorem{remark}[theorem]{Remark}
\def\FINEDIM{\par\penalty10000\hfill$\Box$
\hphantom{mm}
\medskip
	
}
\theoremstyle{plain}
\newtheorem*{namedthm}{\namedthmname}
\newcounter{namedthm}

\makeatletter

\def\P{{\rm P}}

 \newcommand{\R}{\mathbb R}
 
 \newcommand{\C}{\mathbb C}

\newdimen\tinskip
\tinskip=\parindent
\newcommand{\tin}[1]{\par\noindent\hglue\tinskip\hphantom{a)\
}\llap{#1\enspace}\ignorespaces}

\let\dlines=\displaylines

\numberwithin{equation}{section}

\usepackage{hyperref}
\hypersetup{
    bookmarks=true,         
    unicode=false,          
    pdftoolbar=true,        
    pdfmenubar=true,        
    pdffitwindow=false,     
    pdfstartview={FitH},    
    pdftitle={},    
    pdfauthor={},     
    colorlinks=true,       
   linkcolor=blue,          
    citecolor=blue,        
    filecolor=black,      
    urlcolor=blue}           

\begin{document}

\title{Excheangeability and irreducible rotational invariance}
%
\author{Paolo Baldi, Domenico Marinucci, Stefano Trapani}
\address{Department of Mathematics, University of Rome Tor Vergata}

\begin{abstract}
In this note we prove that a finite family $\{X_1,\dots,X_d\}$ of real r.v.'s that is exchangeable and such that $(X_1,\dots,X_d)$ is invariant with respect to a subgroup of $SO(d)$ acting irreducibly, is actually invariant with respect to the action of the full group $SO(d)$. Three immediate consequences are deduced: a characterization of isotropic spherical random eigenfunctions whose Fourier coefficients are exchangeable, an extension of Bernstein's characterization of the Gaussian and a characterization of the Lebesgue measure on the sphere.
\end{abstract}
   \maketitle

\medskip


\footnote{We acknowledge financial support of the MUR Department of Excellence Programme MatModTov, and of the Prin 2022 \emph{Grafia}}

\section{Introduction}
In this paper, we investigate characterizations of exchangeable random variables (r.v.'s) when, in addition, they enjoy some invariance with respect to a group of rotations.

As is well-known, an infinite sequence $\{X_1,X_2,\dots\}$ of random variables on some probability space is exchangeable if and only if for any permutation $\pi:\mathbb{N} \rightarrow \mathbb{N}$ one has that
\[
	\{X_1,\dots ,X_{n},\dots \}\mathop{=}^{Law}\{X_{\pi(1)},\dots ,X_{\pi(n)},\dots \} .
\]
Under this condition, the De Finetti-Hewitt-Savage Theorem implies the existence of a random variable $Y$ on the same probability space such that the r.v.'s $\{X_1,X_2,\dots\}$ are conditionally independent given $Y$. The original statement of De Finetti's Theorem referred to Bernoulli's r.v.'s, but it has been shown to hold for r.v.'s taking values in any Polish space; research in this area is still active, to the point that it is considered by some author a foundational result not only in probability/statistics but for mathematics as a whole (see \cite{Alam} and the references therein).

While the De Finetti-Hewitt-Savage results hold in great generality for infinite sequences, it is also well-known that they fail in general for finite ones. A classical reference by Diaconis and Freedman shows that one can at most obtain upper bounds in total variation between finite-exchangeable sequences and mixture of independent random variables; these bounds are sharp, indeed there exist explicit counterexamples where the bound is saturated, see \cite{DiaFreeAoP}. Another celebrated paper by the same authors (\cite{DiaFreeAIHP}) shows that uniform random vectors on a sphere of growing dimension $n$ any $k$-dimensional set of coordinates suitably normalized is asymptotically close to Gaussian in total variation distance, provided that $k=o(n)$ as $n \rightarrow \infty$.

Our purpose here is to investigate some further characterization of exchangeability for finite-dimensional sequences. In particular, in the next section we prove an algebraic result which implies that invariance with respect to a subgroup of $SO(d)$ acting irreducibly in addition to exchangeability implies invariance with respect to the full group $SO(d)$.

Next, in \S\ref{sec-eig} we deduce properties of the Fourier coefficients of an isotropic random field on the sphere under the assumption that they are exchangeable. {\it En passant} we obtain a characterization of the Lebesgue measure of the spheres $\mathbb{S}^{d-1}$ of $\R^d$; finally in \S\ref{sec-bernstein} we give a proof of a stronger version of Bernstein's Theorem on the characterization of the Gaussian.
%
\section{The main decomposition statement}
Let as consider the special  orthogonal groups $SO(k), SO(N),$ with the corresponding Lie algebras $so(k),so(N)$; assume $k \geq 3, N \geq 3$.
Let
$$
\phi : SO(k) \to SO(N)
$$
be a homomorphism, i.e. a representation of $SO(k)$ on $\R^N$.  Let $S_N$ be the symmetric group of $N$ letters and, for $\sigma \in S_N$ let $E_\sigma$ be the permutation matrix corresponding to the permutation $\sigma$. Then we have the following
main result.
\begin{theorem}\label{mainref} Assume that the representation $\phi$ has no one dimensional invariant subspaces, then  the subgroup of $SO(N)$ generated by the elements $E_\sigma^{-1}  \phi(g) E_\sigma$ with $g \in SO(k)$ and $\sigma \in S_N$ is all of $SO(N)$.

In particular  if $\sigma_0$ is a fixed permutation of sign $-1$ then the subgroup of $O(N)$  generated by
$\{ E_{\sigma_0}, E_\sigma^{-1}  \phi(g) E_\sigma,\ \sigma \in S_N, g \in SO(k) \}$ is all of $O(N)$.
\end{theorem}
\begin{proof}
Recall that the Lie algebras $so(k)$ for $k \geq 3, k \neq 4$ are simple, i.e. they have no non trivial ideals, and that $so(4)$ is the direct sum of two simple three dimensional ideals (see \cite{erdmann} theorem 12.1, and \cite{helg2} p.~352).

Let us first prove that, the image of $\phi$ has necessarily dimension $\ge 3$. If $k\ge 3$ and $k\not=4$;
since the kernel of the differential $d \phi_e$ ($e$ is the identity element of $SO(k)$) is an ideal of $so(k)$, then $\ker \phi$ can only be either $so(k)$ or $\{0\}$. Since $\phi$ has no one dimensional invariant subspaces, the first occurrence cannot happen and $\phi$ is injective. If $k=4$ then $\ker d \phi_e=so(4)$ is impossible as above and $\ker d \phi_e$ can only be $\{0\}$ or one of the two simple three dimensional ideals of $so(4)$, as $so(4)$ has no other proper ideals (see Serre's book \cite{serre}, chapter 6). In any case $\dim \ker d \phi_e\le 3$ so that, as $so(4)$ has dimension $6$, the image of $d\phi_e$ is $\ge 3$ in any case. In particular if $N =3,$ the map $\phi$ is onto, and the proposition is proved, so we may assume $N \geq 4$.

Next, let $\sigma_1, \sigma_2, \ldots, \sigma_{N!}$  be an enumeration of the permutations in $S_N,$ and let us consider the map $F : SO(k)^{N!} \to SO(N),$
$$
F(g_1,g_2, \ldots,g_{N!}) =  (E_{\sigma_1}  \phi(g_1) E_{\sigma_1}^{-1}) (E_{\sigma_2}  \phi(g_2) E_{\sigma_2}^{-1}) \cdots  (E_{\sigma_{N!}}  \phi(g_{N!}) E_{\sigma_{N!}}^{-1})\ .
$$
Note that, if $e_2$ denotes the identity element in $SO(N)$, then $F(e,\ldots,e) = e_2$. Let $W$ be the image  of the differential of $F$ at $(e,\ldots,e)$. Then $W$  is a vector subspace of the tangent space to $SO(N)$ at $e_2$ identified with the Lie algebra $so(N)$ which is the space of real $N \times N$ antisymmetric matrices.  By the Leibnitz rule we have:
$$
W = E_{\sigma_1}d \phi_{e}(so(k)) E_{\sigma_1}^{-1} + (E_{\sigma_2}) d \phi_{e}(so(k)) E_{\sigma_2}^{-1} + \cdots + E_{\sigma_{N!}} d \phi_{e}(so(k)) E_{\sigma_{N!}}^{-1}.
$$
Therefore $W$ is a vector subspace of the Lie algebra $so(N)$ which is invariant under conjugation with respect to any permutation matrix. The remainder of the proof is concerned with showing that $W = so(N)$.
This entails that  that the differential at $(e,\ldots,e)$ of the  map $F$ is onto and in particular the image of $F$ contains a neighborhood of $e$ in $SO(N)$. Since $SO(N)$ is connected, the group generated by a neighborhood of the identity is all of $SO(N)$ and this will conclude the proof.

To this goal we shall take advantage of the decomposition of the complexified Lie algebra $so(N)^{\C}$ under the action of the permutation group $S_N$ that is provided in \cite{can}.

We already know that $W$ is invariant under the representation of $S_N$ $\Lambda_\sigma: A \mapsto E_\sigma A E_{\sigma}^{-1},$, $\sigma \in S_N$, $A \in so(N)$.

Let us consider the  complexified representation $\Lambda^{\mathbb{C}}$ obtained considering the action $\Lambda_\sigma$ on the space  $so(N)^{\mathbb{C}}$ of complex antisymmetric matrices. In \cite{itz} it is proved that, for $N \geq 4$, $\Lambda^{\mathbb{C}}$ decomposes into the direct sum of two irreducible  non isomorphic subrepresentations  $V^{\mathbb{C}}_1,$ $V^{\mathbb{C}}_2$ and the authors  determine the isomorphism class of each irreducible component.

In particular $V^{\mathbb{C}}_1$ is isomorphic to the so called standard representation of $S_N$ i.e. to  the representation of $S_N$ by permutation of the coordinates on the subspace of $\mathbb{C}^N$  $\{z\in\mathbb{C}^N; z_1 + \ldots + z_N = 0\}$, so that  $\dim(V^{\mathbb{C}}_1) = N-1$, and, as
$\dim(so(N))=\frac 12\,N(N-1)$, $\dim(V^{\mathbb{\C}}_2) = \frac12\,(N-1)(N-2)$.

As for $V^{\mathbb{\C}}_2$, let $v_0 =\, ^t\!(1,,\dots,1) \in \mathbb{C}^N$, so that $v_0$ is fixed under the action of $S_N$.
Note that if $v\in \mathbb{C}^N$ is in the kernel of the complex antisymmetric matrix $A$  and $\sigma \in S_N,$ then $ E_\sigma v$ is in the kernel of $\Lambda_\sigma^{\mathbb{C}}A=E_\sigma A E_{\sigma}^{-1}$. Therefore the subspace $U=  \{ A \in so(N)^{\mathbb{\C}} : A v_0 = 0 \}$ is a
$\Lambda^{\mathbb{C}}$-invariant subspace.  Moreover it is a Lie algebra and more precisely the Lie algebra of the subgroup of $SO(N)^{\mathbb{\C}} $ fixing the vector $v_0,$ whose dimension is $\frac12\,(N-1)(N-2)$. If $N > 4$ clearly  $V^{\mathbb{C}}_2=U$. If $N = 4$   the subspaces $V^{\mathbb{C}}_1$ and  $V^{\mathbb{C}}_2$ both have dimension  $3$,  however looking at the character on a two cycle  of the action of $S_4$ on $U$  and on $V^{\mathbb{C}}_1$ we see that again $U = V^{\mathbb{C}}_2$. Details are postponed in the subsequent Lemma \ref{rapp}.

Let $h(A,B) = tr( A ^t\overline{B})$ be the $so(N)$ conjugation-invariant positive Hermitian product. As the orthogonal to  $V^{\mathbb{\C}}_2$ with respect to $h$ is also $\Lambda^\C$-invariant, $V^{\mathbb{C}}_1$ is the $h$-orthogonal subspace  to $V^{\mathbb{\C}}_2$.   Observe that both subspaces $V^{\mathbb{C}}_1$ and $V^{\mathbb{C}}_2$ are invariant with respect to complex conjugation so that the spaces $V_1,V_2$ of their respective real points give rise to two irreducible  real subrepresentations  of $\Lambda$ with  $so(N)  = V_1 + V_2$. As $W$ is itself a $\Lambda$-invariant subspace of $so(N)$, we therefore need to show that $W \nsubseteq V_1$ and $ W \nsubseteq V_2,$ i.e. that $d \phi_e(so(k)) \nsubseteq V_1$ and $d \phi_e(so(k)) \nsubseteq V_2$.

As $V_2$ is the subspace of the matrices that annihilate the vector $v_0=$ $^t\!~(1,\dots,1)$, if $d \phi_e(so(k)) \subseteq V_2$ then the subspace $\mathbb{R} v_0$ of $\mathbb{R}^N$ would be $\phi$ invariant, against the assumptions.

Let us finally prove that also $d \phi_e(so(k)) \subseteq V_1$ cannot be. Let $f_0  =\, ^t\!(1,0,\dots,0)$
and let $B \in SO(N)$ be a matrix such that $B f_0 = \frac1{\sqrt{N}}\,{v_0}$. Since the inner product  $\Re(h)$ on $so(N)$ induced by $h$ is also $SO(N)$ invariant with respect to conjugation by $B$, we see that $B^{-1}V_1B$ and   $B^{-1}V_2B $ are $\Re(h)$-orthogonal.
Now, as the matrices of $V_2$ annihilate $v_0$, $B^{-1}V_2B$ is the set of matrices that annihilate $f_0$, i.e. of the form
$$
\begin{bmatrix}
	0 &   0 \\ 0 &  A
\end{bmatrix}
$$
with $A \in so(N-1)$ and
$B^{-1}V_1B$, being the $h$-orthogonal to $B^{-1}V_2B$, is the set of matrices of the form
$$
\begin{bmatrix} 0 &  -^t\!X \\ X & {\bf{0}}
\end{bmatrix}
$$
with $X \in \mathbb{R}^{N-1}$ and ${\bf{0}}$ is the zero matrix in $so(N-1)$.

Assume by contradiction that  $B^{-1}d \phi_e(so(k)) B \subseteq B^{-1}V_1 B$.

Choose two matrices $C_1,C_2 $ in $B^{-1}d \phi_e(so(k)) B \subseteq B^{-1}V_1B$
with $C_1 \neq 0 $. Since $B^{-1}d \phi_e(so(k)) B$ is a Lie algebra, the  matrices  $C_1, C_2, $ and $ [C_1,C_2] = C_1C_2 - C_2C_1 $ also must belong to  $  B^{-1}V_1 B$. This means that the entries $(C_1)_{i,j}$, $(C_2)_{i,j}$ and $(C_1C_2 - C_2C_1)_{i,j}$ vanish for $i > 1, j > 1$. Choose $i_0 > 0$ such that $(C_1)_{1,i_0} \neq 0$ for some $i_0 > 1$
and fix an index $j$ with $2 \leq j \leq N$. We   have
$$
\dlines{
	0 = [C_1,C_2]_{i_0,j} = (C_1C_2)_{i_0,j}-(C_2C_1)_{i_0,j}\cr
	=\sum_k \bigl((C_1)_{i_0,k} (C_2)_{k,j} -  (C_2)_{i_0,k} (C_1)_{k,j} \bigr)\cr
	=  (C_1)_{i_0,1} (C_2)_{1,j} -  (C_2)_{i_0,1} (C_1)_{1,j}
	\cr
}
$$
which gives
$$
(C_2)_{1,j}  = \frac{(C_2)_{1,i_0} (C_1)_{1,j}}{(C_1)_{1,i_0}}\ \cdotp
$$
Therefore each value  $(C_2)_{1,j}$ is determined by $C_1$ and by $(C_2)_{1,i_0}$.
This means  that the space $B^{-1}d\phi_e(so(k))B$ is one dimensional, but we
know that it is at least three dimensional. This gives the desired contradiction.
\end{proof}
\begin{lemma}\label{rapp} Let $T_1 = \{ Z \in \mathbb{C}^4 : \sum_j z_j = 0 \},$ and consider the representation $\alpha_1$ of the symmetric group $S_4$ on $T_1$ by permutation of the coordinates  (the so called standard representation).

Let $so(4,\mathbb{C})$ be the space of complex antisymmetric $4 \times 4$ matrices, and let $\alpha_2$ be the representation of  $S_4$ on $so(4)$ given by $\alpha_2(\sigma)(A) = E_\sigma A E_\sigma^{-1}$, where $E_\sigma$ is the permutation matrix corresponding to $\sigma \in S_4.$  The subspace $U$ of matrices  in $so(4)$ having the vector $^t(1, \ldots, 1)$ in their kernel is invariant under $\alpha_2.$

Let $\sigma_0$ be the permutation exchanging $1$ and $2$ and keeping $3$ and $4$ fixed. Then the character of $\alpha_1$ on $\sigma_0$ is $1$ and the character of ${\alpha_2}_{|U} $ on $\sigma_0$ is $-1$.
\end{lemma}

\begin{proof}
The value of the character of $\alpha_1$ on $\sigma_0$ is $1$ as shown for example in Fulton-Harris book \cite{fultonharris},  \S 2.3 p.~19).
	
Now if
$$
A =  \begin{pmatrix} 0 &a & b & c \\ - a & 0 & d & e \\ -b & - d & 0 & f \\ -c & -e &-f & 0
\end{pmatrix}
$$
then  the subspace  $U$ of matrices in $so(4)$ which have the  vector $^t(1,\ldots,1)$   in their kernel corresponds to  the vector  subspace $U  \subseteq \mathbb{C}^6$
$$
\displaylines{
	\{ ^t(a,b,c,d,e,f) \in \mathbb{C}^6 : a + b + c = 0, -a  + d + e = 0,\qquad\qquad\qquad\cr
 \qquad\qquad\qquad\qquad\qquad-b - d + f = 0, -c - e - f = 0 \}\cr
	=  \{ ^t(a,b,-a-b,d,a-d,b+d) \in \mathbb{C}^6, ^t(a,b,d) \in \mathbb{C}^3 \}\ . \cr
}
$$
So a basis of $U$ is $ \beta := \{^t(1,0,-1,0,1,0),^t\!(0,1,-1,0,0,1),^t\!(0,0,0,1,-1,1) \}.$
On the other hand for $A$ as above
$$
\alpha_2(\sigma_0)(A) =  \begin{pmatrix}
	0 &-a & d & e \\  a & 0 & b & c \\ -d & - b & 0 & f \\ -e & -c &-f & 0
\end{pmatrix}
$$
Then the $3 \times 3$ matrix $B$ associated to the linear map $\alpha_2(\sigma_0)_{|U}$ with respect to the base $\beta$ is:
$$ B =  \left( \begin{array}{lcc} -1 &0 & 0 \\  0 & 0 & 1 \\ 0 & 1 & 0  \end{array} \right)$$ and $tr(B) = -1.$
\end{proof}
\begin{remark}
Note that any irreducible representation $SO(3) \to GL(N,\mathbb{C})$ is a real representation (see \cite{can}) i.e. there exists a basis of $\mathbb{C}^N$
such that in this basis the representation takes values in $SO(N)$.
\end{remark}
\begin{remark}
Since any semisimple Lie algebra is a direct sum of simple ideals  and any simple Lie algebra has dimension at least $3$, in the above theorem $SO(k)$ can be replaced by any connected semisimple Lie group. G.
\end{remark}
\section{Random eigenfunctions with exchangeable coefficients}\label{sec-eig}
Let $(\Omega,\cl F,\P)$ be a probability space and let us consider a finite-variance, isotropic random fields on the sphere of $\R^3$, i.e. an application $T:\Omega \times \mathbb{S}^2 \rightarrow \mathbb{R}$ such that $\mathbb{E}[T^2]< \infty$ and
\[
	T(g\, \cdot)\enspace\mathop{=}^{\cl L}\enspace T(\cdot) \text{ for all } g \in SO(3)\ ,
\]
where the identity in law holds in the sense of processes; from now on, without loss of generality, we will take all these fields to be zero-mean, i.e. $\mathbb{E}[T(x)]=0$ for every $x\in\mathbb{S}^2$.

Recall that the Laplace-Beltrami operator on $\mathbb{S}^2$ can be written in coordinates as
\[
\Delta_{\mathbb{S}^2}=\frac{1}{\sin \theta}\frac{\partial}{\partial \theta} \sin \theta \frac{\partial}{\partial \theta}+\frac{1}{\sin^2 \theta}\frac{\partial}{\partial \varphi}
\]
with eigenvalues $\lambda_{\ell}=-\ell(\ell+1)$; the corresponding eigenspaces have dimension $2\ell+1$
respectively and the group $SO(3)$ acts irreducibly on them. Note that the ``small'' group $SO(3)$ is able
to act irreducibly on euclidean spaces of arbitrarily large dimension. Their elements are known as spherical harmonics. It is possible to fix for each of them a (real-valued) orthonormal basis, whose elements we will write $Y_{\ell m}:\mathbb{S}^2 \rightarrow \mathbb{R}, \ell=0,1,2,\dots, m=-\ell,\dots,\ell$. There are some standard choices for these elements, most notably the so-called fully normalized spherical harmonics, (see e.g. \cite{dogiocam}, p.~64) but their exact analytic expressions are not relevant for our results, which invariant invariant with respect to the choice of basis.

A standard result in the theory of isotropic spherical random fields is the well-known spectral representation theorem, stating that the following representation holds, in the mean-square sense:
\begin{equation}\label{dev}
T(x)=\sum_{\ell=0}^{\infty}\sum_{m=-\ell}^{\ell}a_{\ell m}Y_{\ell m}(x)\ ,
\end{equation}
for $(a_{\ell m})_{\ell,m}$ a triangular array of real-valued random coefficients with zero mean and uncorrelated, i.e.
\[
	\mathbb{E}[a_{\ell m}a_{\ell' m'}]=C_{\ell}\delta _{\ell}^{\ell'}\delta_{m}^{m'}\ ,
\]
where the nonnegative sequence $(C_{\ell})_{\ell=0,1,\dots}$ is the angular power spectrum of the field and satisfies
\[
	\mathbb{E}[T^2]=\sum_{\ell=1}^{\infty}\frac{2\ell+1}{4 \pi}\,C_{\ell} < \infty .
\]
As the spherical random harmonic coefficients are always uncorrelated for isotropic fields, then they are independent in the Gaussian case. More surprisingly, the converse is also true: i.e., for isotropic random fields, if the spherical random harmonic coefficients are independent, then necessarily they are also Gaussian as well as the field. This fact was shown first in \cite{BM06} and then extended to more general random fields on compact spaces in \cite{MR2342708}, \cite{trap}.

We now provide a characterization of the random spherical harmonic coefficients of isotropic random fields under the assumptions that these coefficients are exchangeable.
Indeed, for $\ell=1,2,\dots$ let  $\{u_{\ell m}\}_{m=-\ell,\dots,\ell}$ denote the components of a sequence of $(2\ell+1)$-dimensional random vectors $u_{\ell.}$, having uniform distribution on the sphere $\mathbb{S}^{2\ell+1}$. The main result of this section is the following:

\begin{theorem}\label{Teorema1} Let $a_{\ell.}=(a_{\ell m})_{m=-\ell,\dots ,\ell}$ be the $(2\ell+1)$-dimensional random vector of the random coefficients of the $\ell$-th component of an isotropic spherical random field. Then, if the $a_{\ell m}$ are exchangeable, their law is rotationally invariant and there exist a nonnegative random variable $\eta_{\ell}$, independent of $a_{\ell.}$, such that $\mathbb{E}[\eta_{\ell}^2]=(2\ell+1)C_{\ell}$ and the following identity in distribution hold, for all $\ell=1,2,\dots $
\begin{equation}\label{decomp}
a_{\ell.}=u_{\ell.} \times \eta_{\ell}\ .
\end{equation}
\end{theorem}
Theorem \ref{Teorema1} is an immediate consequence of Proposition \ref{prop11} below.
\begin{remark}\label{remrem} In \cite{BM06} (see also \cite{MR2342708}, \cite{trap}) the authors consider the set $(Y_{\ell m})_{\ell m}$ of the fully normalized spherical harmonics. These are complex valued functions, so that, in order to obtain a real random field, the coefficients $a_{\ell m}$ must satisfy the relation
$$
\overline{a_{\ell,-m}}=	a_{\ell,m}\ .
$$
The authors prove that, for an isotropic random field, if the coefficients $a_{\ell,0}, \dots,a_{\ell,m}$ are independent, then they are necessarily Gaussian. Note that no assumption of independence is made concerning
the real and the imaginary parts of the $a_{\ell m}$.
This assumption is therefore in some sense not comparable with the exchangeability  of this vector.
	
Note however that when the $(Y_{\ell m})_{\ell m}$ are the fully normalized spherical harmonics, then the complex random r.v. $a_{\ell m}$ is rotationally invariant, so that, for any isotropic random field, the real and imaginary parts of $a_{\ell m}$ are exchangeable r.v.'s.
\end{remark}
\begin{remark}
Due to the results \cite{DiaFreeAIHP} mentioned above, it is clear that in the "high energy" limit $\ell \rightarrow \infty$, for all $k \in \mathbb{N}$ every subset $(a_{\ell, \pi(1)},\dots ,a_{\ell, \pi(k)})$ of cardinality $k=o(\ell)$ of exchangeable spherical harmonic coefficients, suitably normalized, converges in total variation to the mixture of a vector of i.i.d. Gaussian variables with a random standard deviation. In view of these characterizations, one can consider Theorem \ref{Teorema1} as an ``approximate'' De Finetti style result; note that for finite $\ell$ the components of $u_{\ell.}$ are identically distributed, but not independent.
\end{remark}
\begin{prop}\label{prop11} Let $X=^t\!\!(X_1,\dots,X_d)$ be a random vector. Then if $X$ is invariant with respect to a the action of a subgroup of rotations acting irreducibly on $\R^d$ and the family
$\{X_1,\dots,X_d\}$ is exchangeable, then $X$ is rotationally invariant and the r.v.'s $|X|$ and $\frac X{|X|}$ are independent.
\end{prop}
\proof By Theorem \ref{mainref} the law of $X$ is invariant by the action of $SO(d)$ hence the rotational invariance. It is well known that if an $d$-dimensional r.v. $X$  has a law that is rotationally invariant, then the angular component $\frac X{|X|}$ and the radial one $|X|$ are independent. Let us however detail a proof of this fact for thoroughness sake.

Let $A\in\cl B(\mathbb{S}^{d-1})$ and $B\in\cl B(\R^+)$. Then, for every orthogonal matrix $O\in SO(d)$, we have
$$
\P\Bigl(\frac X{|X|}\in A,|X|\in B\Bigr)=\P\Bigl(\frac {OX}{|X|}\in A,|X|\in B\Bigr)
$$
which means that, for every fixed $B\in\cl B(\R^+)$, the finite measure
$$
\cl B(\mathbb{S}^{d-1})\ni A\mapsto \P\Bigl(\frac X{|X|}\in A,|X|\in B\Bigr)
$$
is rotationally invariant, hence a multiple of the normalized Lebesgue measure, $\lambda$ say, of $\mathbb{S}^{d-1}$. The total mass of such a measure is deduced by choosing $A=\mathbb{S}^{d-1}$, i.e. $\P(|X|\in B)$. Therefore
$$
\P\Bigl(\frac X{|X|}\in A,|X|\in B\Bigr)=\lambda(A)\P\bigl(|X|\in B\bigr)=
\P\Bigl(\frac X{|X|}\in A\Bigr)\P\bigl(|X|\in B\bigr)\ .
$$
\FINEDIM
\begin{remark} In general the property of exchangeabilty of the components of a vector $X=^t\!\!(X_1,\dots,X_d)$ is not invariant with respect to rotations, so that {\it a priori} the condition of exchangeability of $X$ depends on the choice of the orthogonal basis that is chosen. However,  as Theorem \ref{Teorema1}, states rotational invariance, it also states that if the coefficients are exchangeable with respect to one basis, then they are also exchangeable with respect to any other basis.
\end{remark}
\begin{remark}  It is well-known that the Lebesgue measure on the sphere $\mathbb{S}^{d-1}$ and its multiples are the unique measures on  $\mathbb{S}^{d-1}$ that are invariant under the action of the rotation group $SO(d)$.

In the same spirit as in the previous section, Proposition \ref{prop11} allows to give a weaker characterization.

In terms of probabilities, if $X=(X_1,\dots,X_d)$ is an $\mathbb{S}^{d-1}$-valued r.v. such that
$\{X_1,\dots,X_d\}$ is exchangeable and whose law, $\mu$ say, is invariant under the action of a subgroup of $SO(d)$ acting irreducibly on $\R^d$, then $\mu$ is the normalized Lebesgue measure on $\mathbb{S}^{d-1}$.
\end{remark}
\section{An extension of Bernstein's theorem} \label{sec-bernstein}
Bernstein's theorem (see \cite{kac-char}, \cite{chaumont-yor} p.74 e.g.) states
that a probability on $\R^d$ which is invariant with respect to the action of
$SO(d)$ and whose components are independent is necessarily Gaussian.
What if we assumed invariance only with respect to the action of a subgroup of $SO(d)$ {\it acting irreducibly} on $\R^d$?
\begin{theorem} \label{bernstein} (Bernstein revisited) Let $X=(X_1,\dots,X_d)$ be a $d$-di\-men\-sio\-nal r.v. such that
\tin{a)} its law is invariant with respect to a subgroup of $SO(d)$ acting irreducibly on $\R^d$;
\tin{b)} the r.v.'s $X_1,\dots, X_d$ have the same law and are independent.

Then $X$ is Gaussian.
\end{theorem}
Theorem \ref{bernstein} follows from Theorem \ref{mainref} and the true Bernstein
Theorem: condition b) actually states that $(X_1,\dots,X_d)$ is exchangeable, hence
invariant with respect to the action of the permutation group so that by Theorem
\ref{mainref} $X$ is invariant with respect to the action of the full group $SO(d)$ and
Bernstein's Theorem allows to conclude.

With respect to the original Bernstein Theorem, Theorem \ref{bernstein} requires a stronger distributional
assumption (equality of the laws of the components $X_i$) but a much
weaker invariance hypothesis, as a subgroup of of $SO(d)$ acting irreducibly
on $\R^d$ can be ``very small'': as noted in \S\ref{sec-eig} the group $SO(3)\subset SO(2d+1)$ can act irreducibly on $\R^{2d+1}$ for every $d$.

\begin{remark}
It should be noted that a result analogous to Theorem \ref{bernstein} follows also from the so-called Skitovich-Darmois Theorem (see e.g. \cite{MR0137201}), which is indeed slightly more general, as it does not require the single components to be equi-distributed. However, our purpose here is not to claim new results, as to show what we consider elegant and simple consequences of our main Theorem \ref{mainref}.
\end{remark}
%
%
\bibliography{bibbase}
\bibliographystyle{amsplain}
\end{document}